\documentclass[12pt,a4paper]{article}
\usepackage[top=3.0cm,bottom=3.0cm,left=2.5cm,right=2.5cm]{geometry}

\usepackage{amssymb,amsmath,graphicx,color,amsthm,centernot,enumerate}
\usepackage[capitalise]{cleveref}
\usepackage[labelformat=simple]{subcaption}
\usepackage[font=small,labelfont=bf,width=12.5cm]{caption}

\usepackage{multirow,tabularx}

\newtheorem{claim}{Claim}
\newtheorem{theorem}{Theorem}[section]

\newtheorem{corollary}[theorem]{Corollary}

\newtheorem{conjecture}[theorem]{Conjecture}
\newtheorem{problem}[theorem]{Problem}
\newtheorem{question}[theorem]{Question}

\newcommand{\set}[1]{\ensuremath{\left\{#1 \right\}}}

\newcommand{\chis}[1]{\ensuremath{\chi_{s}'(#1)}}
\newcommand{\chils}[1]{\ensuremath{\chi_{s,l}'(#1)}}
\newcommand{\chin}[1]{\ensuremath{\chi_{n}'(#1)}}
\newcommand{\chiln}[1]{\ensuremath{\chi_{n,l}'(#1)}}
\newcommand{\coef}[2]{\ensuremath{\mathrm{coef}(#1; \ #2)}}

\newenvironment{proofclaim}[1][]%
    {\noindent \emph{Proof.} {}{#1}{}}{$~$\hfill $~\blacklozenge$ \vspace{0.2cm}}

\newcommand{\claimqed}{$~$\hfill $~\blacklozenge$ \vspace{0.2cm}}

\definecolor{defblue}{rgb}{0.4,0,0.84}
\definecolor{greyblue}{rgb}{0.23,0.4,0.70}
\definecolor{orange}{rgb}{1.0,0.5,0.2}
\definecolor{violet}{rgb}{0.55,0,0.55}

\usepackage{url}
\makeatletter
\g@addto@macro{\UrlBreaks}{\UrlOrds}
\makeatother

\newcolumntype{Y}{>{\centering\arraybackslash}X}
\renewcommand{\arraystretch}{2}

\begin{document}

\title{{\bf List strong and list normal edge-coloring of (sub)cubic graphs}}

\author
{
	Borut Lu\v{z}ar\thanks{Faculty of Information Studies in Novo mesto, Slovenia.} \thanks{Rudolfovo - Science and Technology Centre Novo mesto, Slovenia.} \quad	
	Edita M\'{a}\v{c}ajov\'{a}\thanks{Comenius University, Faculty of Mathematics, Physics and Informatics, Bratislava, Slovakia.} \quad	
	Roman Sot\'{a}k\thanks{Pavol Jozef \v Saf\'{a}rik University, Faculty of Science, Ko\v{s}ice, Slovakia.} \quad
	Diana \v{S}vecov\'{a}\footnotemark[4] \quad
}

\maketitle

{
\begin{abstract}
	A {\em strong edge-coloring} of a graph is a proper edge-coloring, 
	in which the edges of every path of length $3$ receive distinct colors;
	in other words, every pair of edges at distance at most $2$ must be colored differently.
	The least number of colors needed for a strong edge-coloring of a graph is 
	the {\em strong chromatic index}.
	We consider the list version of the coloring and prove 
	that the list strong chromatic index of graphs with maximum degree $3$ is at most $10$.
	This bound is tight and improves the previous bound of $11$ colors.
	
	We also consider the question whether the strong chromatic index and the list strong chromatic index always coincide.
	We answer it in negative by presenting an infinite family of graphs for which the two invariants differ.	
	For the special case of the Petersen graph, we show that its list strong chromatic index equals $7$,
	while its strong chromatic index is $5$.
	Up to our best knowledge, this is the first known edge-coloring for which there are graphs 
	with distinct values of the chromatic index and its list version.
	
	In relation to the above, 
	we also initiate the study of the list version of the normal edge-coloring.
	A {\em normal edge-coloring} of a cubic graph is a proper edge-coloring, 
	in which every edge is adjacent to edges colored with $4$ distinct colors
	or to edges colored with $2$ distinct colors.
	It is conjectured that $5$ colors suffice for a normal edge-coloring of any bridgeless cubic graph
	and this statement is equivalent to the Petersen Coloring Conjecture.	

	It turns out that similarly to strong edge-coloring, 
	list normal edge-coloring is much more restrictive and consequently 
	for many graphs the list normal chromatic index is greater than the normal chromatic index.
	In particular, we show that there are cubic graphs with list normal chromatic index at least $9$,
	there are bridgeless cubic graphs with its value at least $8$,
	and there are cyclically $4$-edge-connected cubic graphs with value at least $7$.
\end{abstract}
}


\section{Introduction}

A {\em strong edge-coloring} of a graph $G$ is a proper edge-coloring in which the edges at distance at most $2$ receive distinct colors.
Here, we define the {\em distance between two edges} in a graph $G$, 
as the distance between their corresponding vertices in the line graph of $G$;
thus, two adjacent edges are at distance $1$, and two non-adjacent edges, which are adjacent to a common edge, are at distance $2$.
The least number of colors for which $G$ admits a strong edge-coloring is called the {\em strong chromatic index}, 
and is denoted by $\chis{G}$. 

In 1985, Erd\H{o}s and Ne\v{s}et\v{r}il~\cite{Erd88} proposed the following conjecture;
in 1990, it was updated to its current form by Faudree et al.~\cite{FauGyaSchTuz90}, 
in order to fit the graphs with an even or odd maximum degree.

\begin{conjecture}[Erd\H{o}s, Ne\v{s}et\v{r}il~\cite{Erd88}]
	\label{conj:Erdos}
	The strong chromatic index of an arbitrary graph $G$ satisfies
	$$	
			\arraycolsep=1.4pt\def\arraystretch{1.2}
			\chi_s'(G) \le \left \{
	\begin{array}{cl}
			\tfrac{5}{4} \Delta(G)^2\,, &\quad \textrm{if }
			\Delta(G) \textrm{ is even}\\
			\tfrac{1}{4}(5\Delta(G)^2 - 2\Delta(G) +1)\,, &\quad \textrm{if }
			\Delta(G) \textrm{ is odd.}
	\end{array} \right.	
	$$
\end{conjecture}\noindent
We are still far from resolving the conjecture in general as the best known upper bound is $1.772\Delta(G)^2$ 
(provided that $\Delta(G)$ is large enough) due to Hurley et al.~\cite{HurJoaKan21}.
However, when limited to graphs of small maximum degree, we know a bit more;
e.g., for graphs with maximum degree $3$ (we refer to them as to {\em subcubic graphs}), the tight upper bound is established.
\begin{theorem}[Andersen~\cite{And92}; Hor\'{a}k, Qing, and Trotter~\cite{HorQinTro93}]
	For any subcubic graph $G$, it holds that
	$$
		\chis{G} \le 10\,.
	$$
\end{theorem}

There are only two known connected bridgeless subcubic graphs that need $10$ colors for a strong edge-coloring:
the Wagner graph (the left graph in Figure~\ref{fig:wagner}) and the complete bipartite graph $K_{3,3}$ with one subdivided edge.
Moreover, there is also no known connected bridgeless subcubic graph on more than $12$ vertices with the strong chromatic index more than $8$,
and based on that also the following, stronger conjecture was proposed.
\begin{conjecture}[Lu\v{z}ar, M\'{a}\v{c}ajov\'{a}, \v{S}koviera, and Sot\'{a}k~\cite{LuzMacSkoSot22}]
	\label{con:str8}
	For any connected bridgeless subcubic graph $G$ on at least $13$ vertices, it holds that
	$$
		\chis{G} \le 8\,.
	$$
\end{conjecture}

On the other hand, the lower bound of $5$ colors for the strong chromatic index of cubic graphs (i.e., $3$-regular graphs)
is attained precisely by the covers of the Petersen graph~\cite{LuzMacSkoSot22}.

\subsection{List strong edge-coloring}

Our research reported in this paper revolves about the list version of the strong edge-coloring of subcubic graphs.
We say that $L$ is a {\em list assignment} for a graph $G$ 
if it assigns a list $L(e)$ of possible colors to each edge $e$ of $G$.
If $G$ admits a strong edge-coloring $\varphi$ such that $\varphi(e) \in L(e)$ for all edges in $E(G)$, 
then we say that $G$ is {\em strong $L$-edge-colorable} or that $\varphi$ is a {\em strong $L$-edge-coloring} of $G$.
The graph $G$ is {\em strong $k$-edge-choosable} if it is strong $L$-edge-colorable for every list assignment $L$, 
where $|L(e)| \ge k$ for every $e \in E(G)$. 
The {\em list strong chromatic index} $\chils{G}$ of $G$ is the least $k$ such that $G$ is strong $k$-edge-choosable.

For graphs with small maximum degrees, a number of results are already known.
Hor{\v{n}}{\'a}k and Wo{\'z}niak~\cite{HorWoz12} showed that for any cycle,
its list strong chromatic index and its strong chromatic index coincide.
Dai et al.~\cite{DaiWanYanYu18} proved that the list strong chromatic index of subcubic graphs is at most $11$, 
and at most $10$ in the case of subcubic planar graphs.
The latter result was later extended to toroidal graphs by Pang et al.~\cite{PanHuoCheWan21}.
For graphs of maximum degree $4$ it was shown that the list strong chromatic index is at most $22$~\cite{ZhaChaHuMaYan20},
and at most $19$ in the case of planar graphs~\cite{CheHuYuZho19}.

In this paper, we generalize the results from~\cite{DaiWanYanYu18} and~\cite{PanHuoCheWan21}
by establishing a tight upper bound for subcubic graphs.
\begin{theorem}
	\label{thm:lst10}
	For any subcubic graph $G$, it holds that
	$$
		\chils{G} \le 10\,.
	$$	
\end{theorem}
As in the non-list version, only
the Wagner graph and the complete bipartite graph $K_{3,3}$ with one subdivided edge
are known to attain the upper bound.

\bigskip
The second question regarding the list strong edge-coloring 
is whether the values of the list strong chromatic index and the strong chromatic index of subcubic graphs coincide.
In particular, we are interested in a special case of the question proposed by Dai et al.~\cite{DaiWanYanYu18}.
\begin{question}[Dai, Wang, Yang, and Yu~\cite{DaiWanYanYu18}, Question~4.1]
	\label{que:dai}
	Is it true that for any graph $G$, it holds that
	$$
		\chils{G} = \chis{G}\,?
	$$
\end{question}

The motivation for the question comes from the List (Edge) Coloring Conjecture stating 
that the values of the chromatic index and the list chromatic index of any graph coincide.
The conjecture was stated independently by several researchers (see~\cite[Problem~12.20]{JenTof95} for more details)
and in general it is still widely open; cf., e.g.,~\cite{BonDelLanPos24} for a short survey.

In other words, the List Coloring Conjecture states that 
the chromatic number of any line-graph is equal to its list chromatic number, 
which is not true for graphs in general. 
Therefore, it seems that the structural properties of line-graphs are the ones that guarantee 
the equality of the two invariants.
One can thus ask what are other structural properties of graphs that would also guarantee equality.
In this sense, Kostochka and Woodall~\cite{KosWoo01} conjectured that
the chromatic number and the list chromatic number are equal for every
square graph, where the {\em square graph $G^2$} is obtained from a graph $G$ by connecting all pairs of vertices at distance $2$.
The conjecture was refuted in general by Kim and Park~\cite{KimPar15}, but it is open for specific graph classes;
for example, whether the two chromatic numbers are equal for the squares of line graphs.
Since a (list) strong edge-coloring of a graph $G$ is exactly a (list) coloring of vertices of the square of the line graph of $G$,
the Question~\ref{que:dai} asks exactly that.
We answer it in negative by presenting an infinite family of graphs $G$ for which $\chils{G} > \chis{G}$.
\begin{theorem}
	\label{thm:infi}
	There is an infinite family of connected cubic graphs $G$ with
	$$
		\chis{G} = 5 \quad {\rm~and~} \quad \chils{G} > 5\,.
	$$
\end{theorem}

Interestingly, there are also some planar graphs (e.g., the dodecahedron) and bipartite graphs (e.g., the generalized Petersen graph $GP(10,3)$)
among the graphs with different values for the two invariants.
Note that the above results are independently obtained also by Hasanvand~\cite{Has22}.

Finally, for the case of the Petersen graph, we prove the exact value of the list strong chromatic index.
\begin{theorem}
	\label{thm:pet7}
	For the Petersen graph $P$, it holds that
	$$
		\chils{P} = 7\,.
	$$
\end{theorem}

After publishing the preprint of this paper, we were notified by M. Hasanvand that the result 
of Theorem~\ref{thm:pet7} follows also from the result of Kierstead on complete multipartite graphs~\cite{Kie00}.

\subsection{List normal edge-coloring}

The second part of this paper is dedicated to initiating the study 
of the list version of the normal edge-coloring of cubic graphs.

A {\em normal edge-coloring} of a cubic graph is a proper edge-coloring, 
in which every edge is adjacent to edges colored with four distinct colors 
(such edges are called {\em rich})
or to edges colored with two distinct colors
(such edges are called {\em poor}).
If at most $k$ colors are used, we call the coloring a {\em normal $k$-edge-coloring}.
The smallest $k$, for which a graph $G$ admits a normal $k$-edge-coloring is the {\em normal chromatic index},
denoted by $\chin{G}$.
Clearly, every strong edge-coloring is also a normal edge-coloring, 
since every edge is rich.
On the other hand, if a cubic graph admits a proper edge-coloring with $3$ colors, 
then every edge is poor, and hence the coloring is also normal.

The normal edge-coloring was defined by Jaeger~\cite{Jae85} as an equivalent way of formulating the Petersen Coloring Conjecture~\cite{Jae88},
which asserts that the edges of every bridgeless cubic graph $G$ can be colored by using the edges of the Petersen graph $P$ as colors
in such a way that adjacent edges of $G$ are colored by adjacent edges of $P$;
in particular, a bridgeless cubic graph admits a normal $5$-edge-coloring if and only if it admits a Petersen coloring.
\begin{conjecture}[Jaeger~\cite{Jae85}]
	\label{con:Nor}
	For any bridgeless cubic graph $G$, it holds that
	$$
		\chin{G} \le 5\,.
	$$
\end{conjecture}
Resolving Conjecture~\ref{con:Nor} would have a huge impact to the theory as it implies two famous conjectures;
namely, the Cycle Double Cover Conjecture~\cite{Jae85b} and the Berge-Fulkerson Conjecture~\cite{Ful71}; cf.~\cite{KarMacZer23} 
for more details.

In general, it is known that every cubic graph (with the bridgeless condition omitted) admits a normal $7$-edge-coloring~\cite{MazMkr20}, 
and the bound is tight, e.g., by any cubic graph that contains as a subgraph the complete graph $K_4$ with one edge subdivided.
When considering only bridgeless cubic graphs, 
Mazzuoccolo and Mkrtchyan~\cite{MazMkr20b} proved that all claw-free cubic graphs, tree-like snarks, and permutation snarks~\cite{MazMkr20b} 
admit a normal $6$-edge-coloring;
the latter result was generalized to bridgeless cubic graphs of oddness $2$ by Fabrici et al.~\cite{FabLuzSotSve24}.
With at most $5$ colors available, only very particular graphs are known to admit a normal edge-coloring, 
see, e.g.,~\cite{FerMazMkr20,HagSte13,SedSkr24b,SedSkr24}.
Hence, Conjecture~\ref{con:Nor} remains widely open in general.

\bigskip
In this paper, in relation to the list strong edge-colorings, 
we also study the properties of the list version of the normal edge-coloring.
For a cubic graph $G$, list normal edge-coloring and the {\em list normal chromatic index}, $\chiln{G}$,
are defined analogously to the list strong variants.

Clearly, the upper bound for the list normal chromatic index of cubic graphs is implied by Theorem~\ref{thm:lst10}.
\begin{corollary}
	\label{cor:lstnor}
	For any subcubic graph $G$, it holds that
	$$
		\chiln{G} \le 10\,.
	$$		
\end{corollary}

We show that, similarly to the list strong edge-coloring, also in the list normal edge-coloring there are graphs $G$
with $\chiln{G} > \chin{G}$.
In particular, 
there is an infinite family of cubic graphs with list normal chromatic index at least $9$,
there are bridgeless cubic graphs with list normal chromatic index at least 8,
and 
there is an infinite family of cyclically 4-edge-connected cubic graphs with list normal chromatic index 
at least~7.
Interestingly, our examples of bridgeless graphs for the above results are all from class I,
and therefore they all have the normal chromatic index equal to $3$.

The paper is structured as follows. 
In Section~\ref{sec:prel}, we introduce notation, terminology, and auxiliary results.
In Sections~\ref{sec:upper} and~\ref{sec:strlow}, we prove results regarding the list strong chromatic index,
and in Section~\ref{sec:norlow}, we present constructions of graphs
with distinct normal and list normal chromatic indices.
We conclude the paper with some open problems in Section~\ref{sec:conc}.

%%%%%%%%%%%%%%%%%%%%%%%%%%%%%%%%%%%%%%%%%%%%%%%%%%%%%%%%%%%%%%%%%%%%%%%%%%%%%
\section{Preliminaries}
\label{sec:prel}

In this section, we introduce the terminology and auxiliary results used in the paper.

As usual, for a sequence of consecutive integers, we use the abbreviation $[i,j] = \set{i,i+1,\dots,j}$.
We call a cycle of length $k$ a {\em $k$-cycle}.
The {\em edge-neighborhood $N(e)$} of an edge $e$ is the set of edges adjacent to $e$,
and the {\em $2$-edge-neighborhood $N_2(e)$} is the set of edges at distance $1$ or $2$ from $e$.
An {\em induced matching} is a set of edges $M$ such that any pair of edges in $M$ is at distance at least $3$;
i.e., the graph induced on the endvertices of the edges of $M$ is a matching.

For a given list assignment $L$, a {\em partial strong $L$-edge-coloring} $\varphi$ of a graph $G$ 
is a strong edge-coloring of a subset of edges of $G$ such that any pair of colored edges $e$ and $f$;
i.e., we have $\varphi(e) \in L(e)$, $\varphi(f) \in L(f)$ and $\varphi(e) \neq \varphi(f)$ 
if $e$ and $f$ are at distance at most $2$ in $G$.

Given a list assignment $L$ and a partial strong $L$-edge-coloring,
we say that a color $c \in L(e)$ is {\em available} for the edge $e$
if no edge in $N_2(e)$ is colored with $c$.
We denote the set of all available colors for an edge $e$ with $A(e)$. Clearly, $A(e) \subseteq L(e)$.

In our proofs, we use the following application of Hall's Marriage Theorem~\cite{Hal35}.
\begin{theorem}
	\label{thm:hall}
	Let $G$ be a graph and $\varphi$ a partial (strong) edge-coloring of $G$.
	Let $X = \set{e_1,\dots,e_k}$ be the set of non-colored edges of $G$. 
	Let $\mathcal{F} = \set{A(e_1),\dots,A(e_k)}$. 
	If for every subset $\mathcal{X} \subseteq \mathcal{F}$ it holds that
	$$
		|\mathcal{X}| \le  \Big | \bigcup_{X \in \mathcal{X}} X \Big |\,,
	$$
	then one can choose an available color for every edge in $X$ such that all the edges receive distinct colors.
\end{theorem}

One of the strongest tools for determining whether colors from the sets of available colors 
can always be found such that the given conditions are satisfied is the following result due to Alon~\cite{Alo99}.
\begin{theorem}[Combinatorial Nullstellensatz~\cite{Alo99}]
	\label{thm:null} 
	Let $\mathbb{F}$ be an arbitrary field, 
	and let $P=P(X_1,\ldots,X_n)$ be a polynomial in $\mathbb{F}[X_1,\ldots,X_n]$.
	Suppose that the coefficient of the monomial $X_1^{k_1}\ldots X_n^{k_n}$, 
	where each $k_i$ is a non-negative integer, 
	is non-zero in $P$ 
	and the degree ${\rm deg}(P)$ of $P$ equals $\sum_{i=1}^n k_i$.
	If moreover $S_1,\ldots,S_n$ are any subsets of $\mathbb{F}$ with $|S_i|>k_i$ for $i=1,\ldots,n$,
	then there are $s_1\in S_1,\ldots,s_n\in S_n$ such that $P(s_1,\ldots,s_n) \neq 0$.
\end{theorem}

In short, for $P_G$ being the graph polynomial of a graph $G$, 
if there is a monomial $m$ of $P_G$ with degree $\deg(P_G)$ and a non-zero coefficient, 
and moreover in $m$ the degree of every variable is less than the number of available colors for 
the vertex represented by the variable,
then there exists a coloring of $G$.
For a monomial $m$, we denote the coefficient of $m$ in the polynomial $P_G$ by $\coef{P_G}{m}$.

Usually, we only consider edge-coloring of a subgraph $H$ of a graph $G$, 
with some of the other edges in $G$ already being precolored 
and hence the lists of available colors for edges in $H$ are reduced accordingly.
In order to apply Theorem~\ref{thm:null}, we construct an auxiliary conflict graph $C(H)$,
in which every vertex represents an edge to be colored,
and two vertices are adjacent whenever the corresponding edges need to be colored with distinct colors. 
Clearly, the input to Theorem~\ref{thm:null} is the graph polynomial of $C(G)$, 
but to avoid this step, we simply say that we consider a {\em conflict graph polynomial} for $H$.

Note that in this paper, every conflict graph polynomial is homogeneous,
i.e., it is a sum of monomials of the same degree,
and therefore the degree condition of Theorem~\ref{thm:null} for monomials is always fulfilled.

%%%%%%%%%%%%%%%%%%%%%%%%%%%%%%%%%%%%%%%%%%%%%%%%%%%%%%%%%%%%%%%%%%%%%%%%%%%%%
\section{Upper bound on the list strong chromatic index}
\label{sec:upper}

In this section, we prove the tight upper bound for the list strong chromatic index.

In the first part of our proof, 
we follow the proof of the result of Dai et al.~\cite{DaiWanYanYu18}
that the list strong chromatic index of subcubic graphs is at most $11$.
In particular, they showed that for eliminating cycles of length at most $5$
from the minimal counterexample, one can even assume lists of length $10$.

\begin{proof}[Proof of Theorem~\ref{thm:lst10}.]
	Suppose the contrary and let $G$ be a minimal counterexample to the theorem;
	i.e., a graph with maximum degree $3$, which has the list strong chromatic index greater than~$10$.
	
	Clearly, $G$ is connected. 
	Moreover, from~\cite{DaiWanYanYu18}, 
	we have the following structural properties of $G$ (since lists of size $10$ are assumed in these lemmas).

	\begin{claim}[{\cite[Lemma~2.1]{DaiWanYanYu18}}]
		\label{cl:reg}
		$G$ is $3$-regular. \claimqed
	\end{claim}

	\begin{claim}[{\cite[Lemma~2.2]{DaiWanYanYu18}}]
		\label{cl:no3}
		$G$ does not contain any $3$-cycle. \claimqed
	\end{claim}	
	
	\begin{claim}[{\cite[Lemma~2.3]{DaiWanYanYu18}}]
		\label{cl:no4}
		$G$ does not contain any $4$-cycle. \claimqed
	\end{claim}		
	
	\begin{claim}[{\cite[Lemma~2.4]{DaiWanYanYu18}}]
		\label{cl:no5}
		$G$ does not contain any $5$-cycle. \claimqed
	\end{claim}			
	
	Next, we reduce cycles of length at least $6$.
	
	\begin{claim}
		\label{cl:6cyc}
		$G$ does not contain any $6$-cycle.
	\end{claim}	
	
	\begin{proofclaim}
		Suppose the contrary and let $C=v_0\dots v_5$ be a $6$-cycle in $G$.
		For every $i \in \set{0,\dots,5}$, call the edge $x_i = v_iv_{i+1}$ (indices modulo $6$) a {\em cycle edge},
		and every non-cycle edge $y_{i}$ incident to $v_i$ a {\em pendant edge} (see Figure~\ref{fig:C6}).
		\begin{figure}[htp!]
			$$
				\includegraphics{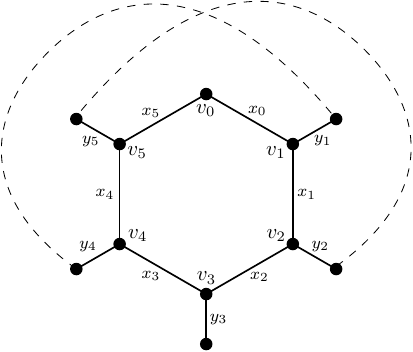}
			$$
			\caption{The hypothetical $6$-cycle $C$ in $G$.
				The edges $y_1$ and $y_4$ (and also $y_2$ and $y_5$) might be connected by an edge (depicted dashed).}
			\label{fig:C6}
		\end{figure}			
		
		By the minimality of $G$, there exists a list strong edge-coloring $\varphi'$ 
		of $G' = G \setminus \set{v_1,v_2,v_3,v_4,v_5}$ for any list assignment $L$ with lists of size at least $10$.
		Let $\varphi$ be the coloring of $G$ induced by $\varphi'$.
		Then, only the edges of $C$ and the pendant edges except $y_0$ are non-colored in $\varphi$.
		The edges $x_0, x_1, x_4$ and $x_5$ have at least $5$ available colors,
		the edges $x_2$ and $x_3$ have at least $6$, 
		$y_1$ and $y_5$ have at least $3$, 
		and $y_2,y_3,y_4$ have at least $4$ available colors.
		
		Claims~\ref{cl:no3}--\ref{cl:no5} imply that no two pendant edges are the same or adjacent;
		it may however happen that the edges $y_1$ and $y_4$ (and similarly, $y_2$ and $y_5$) are connected by an edge;
		we thus assume also these two edges.
		So, the conflict graph polynomial $P_{C_6}$ created on the non-colored edges with conflicts 
		between edges at distance at most $2$ is the following (taking indices modulo $6$):		
		\begin{align*}
			P_{C_6}(x_0,\dots,x_5,y_1,\dots,y_5) =& \Bigg [ \prod_{i=0}^5 (x_i - x_{i+1}) \cdot (x_i - x_{i+2}) \Bigg ] \\
				 & \cdot (x_0 - y_1) \cdot (x_0 - y_2) \cdot (x_0 - y_5) \\
				 & \cdot (x_1 - y_2) \cdot (x_1 - y_3) \cdot (x_1 - y_1) \\
				 & \cdot (x_2 - y_3) \cdot (x_2 - y_4) \cdot (x_2 - y_2) \cdot (x_2 - y_1) \\
				 & \cdot (x_3 - y_4) \cdot (x_3 - y_5) \cdot (x_3 - y_3) \cdot (x_3 - y_2) \\
				 & \cdot (x_4 - y_5) \cdot (x_4 - y_4) \cdot (x_4 - y_3) \\
				 & \cdot (x_5 - y_1) \cdot (x_5 - y_5) \cdot (x_5 - y_4) 	\\
				 &\cdot (y_1 - y_2) \cdot (y_2 - y_3) \cdot (y_3 - y_4) \cdot (y_4 - y_5) \\
				 & \cdot (y_1 - y_4) \cdot (y_2 - y_5)
		\end{align*}
		
		Using the function {\tt Coefficient} in Wolfram Mathematica~\cite{Mat24}, we infer that 
		in $P_{C_6}$, we have the coefficient 
		$$
			\coef{P_{C_6}}{x_0^4 \ x_1^4 \ x_2^5 \ x_3^5 \ x_4^4 \ x_5^4 \ y_1^2 \ y_2^3 \ y_3^2 \ y_4^3 \ y_5^2} = -2\,,
		$$
		which, by Theorem~\ref{thm:null}, 
		means that we can extend the coloring $\varphi$ to all the edges of $G$, a contradiction.
	\end{proofclaim}

	We continue by showing that in $G$ any cycle is reducible.
	\begin{claim}
		\label{cl:nocyc}
		$G$ does not contain any cycle.
	\end{claim}				
	
	\begin{proofclaim}
		Let $C = v_0\dots v_{k-1}$ be a shortest cycle in $G$.
		For every $i \in \set{0,\dots,k-1}$, call the edge $x_i = v_iv_{i+1}$ (indices modulo $k$) a {\em cycle edge},
		and every non-cycle edge $y_{i+1}$ incident to $v_i$ a {\em pendant edge} (see Figure~\ref{fig:kcyc}).		
		By Claims~\ref{cl:no3}--\ref{cl:6cyc}, we have that $k \ge 7$.
		Moreover, since there is no $(k-1)$-cycle in $G$, we have that no pair of pendant edges is connected by any edge
		except by a cycle edge.
		\begin{figure}[htp!]
			$$
				\includegraphics{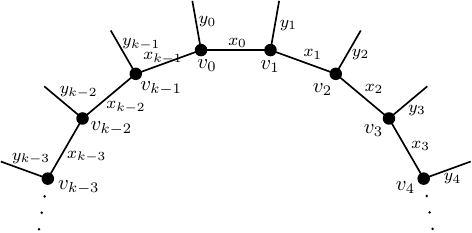}
			$$
			\caption{The hypothetical $k$-cycle $C$ in $G$.}
			\label{fig:kcyc}
		\end{figure}					
		
		Let $L$ be a list assignment for the edges of $G$ with lists of size at least $10$ for which $G$ is not strongly $L$-edge-choosable.
		Let $G'$ be the graph obtained from $G$ by removing the vertices of $C$.
		By the minimality, $G'$ admits a list strong edge-coloring $\varphi'$ with color of every edge $e \in E(G')$ from $L(e)$.
		Let $\varphi$ be the coloring of $G$ induced by $\varphi'$, where only the edges incident to the vertices in $V(C)$
		are non-colored.
		In particular, every cycle edge $x_i$ has at least $6$ available colors, and every pendant edge $y_i$ has at least $4$.
		We will show that we can extend $\varphi$ to all the edges of $G$.
		
		First, let $P_{C_k}$ be the conflict graph polynomial created on the non-colored edges with conflicts 
		between edges at distance at most $2$; taking indices modulo $k$, we have the following:
		\begin{align*}
			&P_{C_k}(x_0,\dots,x_{k-1},y_0,\dots,y_{k-1}) = \\ 
			&\prod_{i=0}^{k-1} (x_i - x_{i+1}) \cdot (x_i - x_{i+2}) \cdot (y_i - y_{i+1}) \cdot (x_i - y_{i-1}) \cdot (x_i - y_{i}) \cdot (x_i - y_{i+1}) \cdot (x_i - y_{i+2})\,.
		\end{align*}		
		Next, we prove that
		$$
			\coef{P_{C_k}}{x_0^4y_0^3 \ x_1^2y_1^2 \ x_2^5y_2^3 \ x_3^5y_3^3 \ x_4^4y_4^3 \ x_5^5y_5^3 
			\cdot \prod_{i=6}^{k-1} x_i^4y_i^3} = (-1)^{k}\,.
		$$		
		
		In order to compute the coefficient in general, we use Wolfram Mathematica~\cite{Mat24}.
		Due to the limitations of the software, we need to split our computations into several steps;
		in particular, we compute coefficients of selected subpolynomials.
		
		We begin by considering the subpolynomial $P_{C_k}^{2,3,4}$ of $P_{C_k}$, comprised of all factors containing
		$x_i$ or $y_i$ for $i \in \set{2,3,4}$.
		The polynomial $P_{C_k}^{2,3,4}$ has degree $29$ and we infer that
		$$
			\coef{P_{C_k}^{2,3,4}}{x_2^5y_2^3x_3^5 y_3^3x_4^4 y_4^3} = x_0^2 x_5^2 y_1^2 + 2 x_0^2 x_5 y_1^2 y_5 - x_0^2 x_5^2 y_5^2 - 2 x_0 x_5^2 y_1 y_5^2 + x_0^2 y_1^2 y_5^2 - x_5^2 y_1^2 y_5^2 \,.
		$$
		Here and in several subsequent cases, we slightly abuse the notation as the value of the coefficient is a polynomial, 
		which appears in the conflict graph polynomial multiplied with the monomial given as an argument.
		Note that in the resulting polynomial, no variable from the monomial appears.
		
		In the second step, we create polynomial $P_{C_k}^{5,6}$, comprised of $\coef{P_{C_k}^{2,3,4}}{x_2^5y_2^3x_3^5 y_3^3x_4^4 y_4^3}$ and multiplied with 
		all factors containing $x_i$ or $y_i$ for $i \in \set{5,6}$,
		which were not yet used in $P_{C_k}^{2,3,4}$.
		We infer that
		$$
			\coef{P_{C_k}^{5,6}}{x_5^5 y_5^3 x_6^4y_6^3} = x_0^2 y_1^2 (x_7 + y_7)\,.
		$$
		Therefore, $x_0^2y_1^2(x_7+x_7)$ is also the coeficient of the monomial $x_2^5y_2^3x_3^5 y_3^3x_4^4 y_4^3 x_5^5 y_5^3 x_6^4y_6^3$ 
		in the subpolynomial of $P_{C_k}$ containing $x_i$ or $y_i$ for all $i \in \set{2,\dots,6}$.

		% todo - ordering
		Now, we define (again, indices modulo $k$) an auxiliary polynomial 
		\begin{align*}
			A_i(x_i,x_{i+1},x_{i+2},y_i,y_{i+1},y_{i+2}) =& (x_i-x_{i+1})(x_i-x_{i+2}) (x_i-y_i) \cdot \\
														 & \cdot (x_i-y_{i+1})(x_i-y_{i+2}) (x_{i+1}-y_i) (y_i-y_{i+1})\,,
		\end{align*}
		used for defining partial polynomials for each of the remaining pairs $x_i$, $y_i$.
		Let
		$$
			P_{C_k}^{7}(x_0,x_7,x_8,x_9,y_1,y_7,y_8,y_9) = \coef{P_{C_k}^{5,6}}{x_5^5 y_5^3 x_6^4y_6^3}\cdot A_7(x_7,x_{8},x_{9},y_7,y_{8},y_{9})\,.			
		$$
		Then,
		$$
			\coef{P_{C_k}^{7}}{x_7^4y_7^3} = -x_0^2 y_1^2 (x_8 + y_8)\,.
		$$
		Finally, for every $i$, $8 \le i \le k-1$, let
		$$
			P_{C_k}^{i}(x_0,x_i,x_{i+1},x_{i+2},y_1,y_i,y_{i+1},y_{i+1}) = (-1)^{i} \cdot x_0^2 y_1^2 (x_{i} + y_{i}) \cdot A_{i}(x_i,x_{i+1},x_{i+2},y_i,y_{i+1},y_{i+1})\,,
		$$
		obtaining
		$$
			\coef{P_{C_k}^{i}}{x_i^4y_i^3} = (-1)^{i} \cdot x_0^2 y_1^2 (x_{i+1} + y_{i+1})\,.
		$$
		
		In the last step, we consider the non-used factors with $x_0$, $x_1$, $y_0$, and $y_1$; we have
		\begin{align*}
			P_{C_k}^{0}(x_0,x_1,y_0,y_1) &= \coef{P_{C_k}^{k-1}}{x_{k-1}^4y_{k-1}^3} \cdot \\
										 &\cdot (x_0-x_1)(x_0-y_0)(x_0-y_1)(x_1-y_0)(x_1-y_1)(y_0-y_1)\,,
		\end{align*}
		giving us
		$$
			\coef{P_{C_k}^{0}}{x_0^4y_0^3} = (-1)^k \cdot (x_1^2 y_1^2-y_1^4) \,.
		$$
		
		This means that
		$$
			\coef{P_{C_k}}{x_0^4y_0^3 \ x_1^2y_1^2 \ x_2^5y_2^3 \ x_3^5y_3^3 \ x_4^4y_4^3 \ x_5^5y_5^3 
			\cdot \prod_{i=6}^{k-1} x_i^4y_i^3} = (-1)^{k}\,,
		$$
		which implies, by Theorem~\ref{thm:hall}, that we can always extend the coloring $\varphi$ to all the edges of $G$, a contradiction.
	\end{proofclaim}	
		
	Since $G$ must be $3$-regular by Claim~\ref{cl:reg}, 
	but it does not contain any cycle by Claims~\ref{cl:no3} to~\ref{cl:nocyc}, we obtain a contradiction establishing the theorem.
\end{proof}

%%%%%%%%%%%%%%%%%%%%%%%%%%%%%%%%%%%%%%%%%%%%%%%%%%%%%%%%%%%%%%%%%%%%%%%%%%%%%

\section{Graphs $G$ with $\chis{G} < \chils{G}$}
\label{sec:strlow}

In this section, we give a negative answer to Question~\ref{que:dai}
by proving Theorem~\ref{thm:infi}.
First, we recall the result about cubic graphs with strong chromatic index equal to $5$.
It uses the notion of covering graphs defined as follows.
A surjective graph homomorphism $f\colon \tilde G\to G$ is
called a \textit{covering projection} if for every vertex
$\tilde v$ of $\tilde G$ the set of edges incident with $\tilde v$ 
is bijectively mapped onto the set of edges incident with $f(\tilde v)$. 
The graph $G$ is usually referred to as the \textit{base graph} and $\tilde G$ as a
\textit{covering graph} or a \textit{lift} of $G$. 
A graph $\tilde G$ \textit{covers} $G$ if there exists such a covering projection.

\begin{theorem}[Lu\v{z}ar, M\'{a}\v{c}ajov\'{a}, \v{S}koviera, and Sot\'{a}k~\cite{LuzMacSkoSot22}]
	\label{thm:martin}
	The strong chromatic index of a cubic graph $G$ equals $5$ 
	if and only if $G$ covers the Petersen graph.
\end{theorem}

Let $P'$ be the graph obtained from the Petersen graph by replacing one edge 
with two pendant edges (see Figure~\ref{fig:peter2cut}).
Consider the labeling of its vertices as given in the figure.
For $1\le i\le 5$, we call the edges $u_iv_i$ the {\em spokes} of $P'$,
the edges $u_iu_{i+2}$ (indices modulo $5$) the {\em inner edges},
and the edges $v_iv_{i+1}$ (indices modulo $5$ and $i=1$ skipped) the {\em outer edges}.
\begin{figure}[htp!]
	$$
		\includegraphics{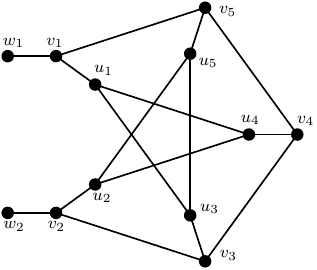}
	$$
	\caption{The graph $P'$ obtained from the Petersen graph by replacing one edge with two pendant edges.}
	\label{fig:peter2cut}
\end{figure}		

%\begin{proposition}
%	\label{prop:peter2cut}
%	In any strong $5$-edge-coloring of $P'$, 
%	the two pendant edges receive the same color.
%\end{proposition}

%\begin{proof}
%	Consider the graph $P'$ with the labeling of vertices as given in Figure~\ref{fig:peter2cut}.
%	Suppose the contrary and let $\varphi$ be a strong $5$-edge-coloring of $P'$
%	such that $\varphi(v_1w_1) \neq \varphi(v_2w_2)$.
%	Then, since $P'$ has $16$ edges,
%	there must be at least one induced matching of size at least $4$ in $P'$,
%	which does not contain both pendant edges.
%	It is an easy exercise to verify that such a matching does not exist.
%\end{proof}

We are now ready to prove Theorem~\ref{thm:infi}.

\begin{proof}[Proof of Theorem~\ref{thm:infi}.]
	Let $R$ be a covering graph of the Petersen graph $P$.
	By Theorem~\ref{thm:martin}, we have that $\chis{R} = 5$;
	let $\varphi_R$ be a strong $5$-edge-coloring of $R$.
	
	Consider the graph $G$ obtained from $R - uv$ (for some edge $uv$ of $R$) and $P'$
	by identifying the vertices $u$ and $w_1$, 
	and $v$ and $w_2$.
	
	We first show that $\chis{G} = 5$.
	Let $\pi$ be a strong $5-$edge-coloring of $P'$.
	with the two pendant edges having the same color.
	We obtain a strong $5$-edge-coloring $\varphi$ of $G$
	by keeping the colors from $\varphi_R$ on the edges of $R-uv$,
	setting $\varphi(uv_1) = \varphi(vv_2) = \varphi_R(uv)$,
	permuting the colors of $\pi$ such that $\varphi_R(uv) = \pi(v_1w_1)$ and 
	such that the colors on the edges incident to $u$ ($v$) in $\varphi_R$ 
	are distinct from the colors incident to $v_1$ ($v_2$) in $\pi$ (this can be done, since the same color $c$
	of the two pendant edges guarantees that $c$ is the only color incident to both vertices $v_1$ and $v_2$),
	and finally setting $\varphi'(e') = \pi(e)$ for every edge $e' \in E(G)$ that corresponds to an edge $e \in E(P')$.
	Note that, by Theorem~\ref{thm:martin}, this means that $G$ is also a covering graph of $P$.
	
	\bigskip
	Next, we show that $\chils{G} > 5$.
	Let $L$ be a list assignment for $G$ such that
	$L(e) = \set{1,2,3,4,5}$ for every edge $e$ of $G$ corresponding to an inner edge of $P'$,
	$L(e) = \set{1,2,3,4,6}$ for every edge $e$ of $G$ corresponding to a spoke of $P'$,
	and 
	$L(e) = \set{1,2,3,5,6}$ for the remaining edges of $G$.
	
	Let $G'$ be the graph obtained from $G$ by removing all the edges of $P'$ except $uv_1$ and $vv_2$.
	Clearly, $G'$ is the graph $R$ with one edge removed and replaced with two pendant edges,
	and thus it admits a strong $5$-edge-coloring $\varphi^*$ induced by the coloring $\varphi$ of $R$.
	Note that in $\varphi$, the edges $uv_1$ and $vv_2$ receive the same color (the color $\varphi(uv)$).
	Now, we show that in any strong $5$-edge-coloring of $G'$ these two edges must be colored with the same color.
	Since the edges of $G'$ in $L$ have the same lists of size $5$, 
	this will imply that the two edges must receive the same color in any strong $L$-edge-coloring.
	
	First, observe that in $\varphi^*$ the only common color the vertices $u$ and $v$ are incident with
	is color $a=\varphi(uv)$. Let $b \neq a$ be a color incident with $u$, and $c \neq a$ a color incident with $v$.
	Let $k$ be the number of edges of color $b$ in $\varphi^*$.
	Since every edge of color $b$ is adjacent to edges of all other four colors, 
	the edges of $G'$ colored with $b$ in $\varphi^*$ together with their adjacent edges 
	cover all the edges of $G'$ (every edge exactly once) except $vv_2$; we denote this (almost) covering $C_b$.
	Similarly, the edges of $G'$ colored with $c$ in $\varphi^*$ together with their adjacent edges 
	cover all the edges of $G'$ (every edge exactly once) except $uv_1$; we denote this (almost) covering $C_c$.
	
	Now, let $\sigma$ be a strong $5$-edge-coloring of $G'$.
	On the edges of $C_b$, every color appears $k$ times, so together with the edge $vv_2$, 
	the color $\sigma(vv_2)$ appears $k+1$ times.
	Similarly we deduce that using the covering $C_c$, the color $\sigma(uv_1)$ appears on $k+1$ edges,
	and so $\sigma(uv_1) = \sigma(vv_2)$.
	Hence, in every strong $5$-edge-coloring of $G'$ the edges $uv_1$ and $vv_2$ must be colored with the same color.	
	
	Similarly, we can show that around the vertices $u$ and $v$ in $G$, all five colors appear 
	(i.e., the only common incident color is the color of the edges $uv_1$ and $vv_2$).
	Observe that the edges of $G'$ colored with $a$ in $\varphi^*$ (except the edges $uv_1$ and $vv_2$)
	together with their adjacent edges 
	cover all the edges of $G'$ (every edge exactly once) 
	except the edges incident with $u$ and $v$; we denote this covering $C_a$.
	Again, in every strong $5$-edge-coloring of $G'$, on the edges of $C_a$
	every color appears $k-1$ times, while in the whole graph every color appears on $k$ edges, 
	except the color of $uv_1$ and $vv_2$, which appears $k+1$ times.
	This means, that $u$ and $v$ together are incident with edges of all five colors.
	
	Now consider the coloring of the edges of $P'$.
	Clearly, in any strong $L$-edge-coloring,
	all the five colors from $\set{1,2,3,4,5}$ appear on the inner edges of $P'$.
	Similarly, since every spoke edge of $P'$ sees $4$ distinct colors on the inner edges of $P'$,
	every spoke edge can be colored with precisely one of the colors 
	from $\set{1,2,3,4}$ or color $6$, 
	except for the spoke edge that does not have color $5$ in its $2$-edge-neighborhood---that edge must be colored with $6$.
	Moreover, the edges $v_1w_1$ and $v_2w_2$ must be colored with the same color, so that we can combine the colorings of $G'$ and $P'$.
	
	There are three non-isomorphic possibilities on which inner edge color $5$ appears.
	First, suppose that $u_1u_4$ is colored with $5$.
	Then, $u_5v_5$ must be colored with $6$ and therefore $v_2v_3$ is the only outer edge of $P'$
	which can be colored with $5$ or $6$. 
	Therefore the colors $1$, $2$, and $3$ must be used on the remaining outer edges, 
	and consequently $v_1w_1$ and $v_2w_2$ must also both be colored with either $5$ or $6$.
	This is not possible, since $v_1w_1$ has both colors in it<s $2$-edge-neighborhood.	
		
	Second, suppose that $u_3u_5$ is colored with $5$.
	Then, $u_4v_4$ must be colored with $6$ and thus no outer edge of $P'$ can have color $6$.
	Since every outer edge of $P'$ also has color $5$ in the $2$-edge-neighborhood, 
	it follows that the remaining four outer edges must be colored with colors $1$, $2$, and $3$.
	This means that $v_1v_5$ and $v_2v_3$ receive the same color, say $1$.
	But then, color $1$ cannot be incident with $u$ and $v$,
	and consequently, $u$ and $v$ together will not be incident with all five colors, which is not possible by the argument above.
	
	So, we may assume that $u_1u_3$ is colored with $5$.
	Then, $u_2v_2$ must be colored with $6$, and $v_1w_1$ and  $v_2w_2$ must both be colored with the same color as $u_3u_5$, which cannot be $4$---say it is $1$.
	Then, the outer edges of $P'$ must be colored with colors from $\set{2,3,5,6}$.
	Since only $v_4v_5$ can be colored with $5$, it follows that $v_1v_5$ must have color $6$.	
	Therefore, $u_2v_2$ and $v_1v_5$ both have color $6$, which means that, 
	since $uv_1$ and $vv_2$ both have color $1$, 
	some color, different $1$, must be incident with $u$ and $v$.
	As we showed above, this is not possible, 
	and therefore a strong $L$-edge-coloring of $G$ does not exist.
\end{proof}

As already mentioned, there are planar graphs and bipartite graphs with different values of the strong chromatic index and the list strong chromatic index.
Two representative examples are the dodecahedron and the generalized Petersen graph $GP(10,3)$ (see Figure~\ref{fig:dodeca}).
\begin{figure}[htp!]
	$$
		\includegraphics{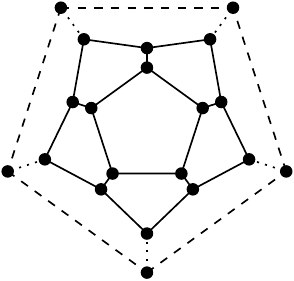}\qquad
		\includegraphics{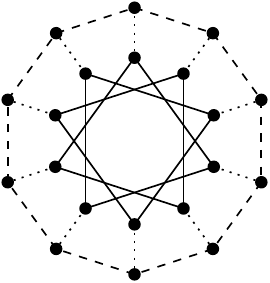}
	$$
	\caption{The dodecahedron (left) and the generalized Petersen graph $GP(10,3)$ (right).}
	\label{fig:dodeca}
\end{figure}
Both graphs cover the Petersen graph and thus their strong chromatic indices are $5$,
while neither of them is colorable from the list assignment assigning 
the list $\set{1,2,3,4,5}$ to the solid edges,
the list $\set{1,2,3,4,6}$ to the dotted edges,
and the list $\set{1,2,3,5,6}$ to the dashed edges, as they are depicted in the figure.
We omit the proof.

Theorem~\ref{thm:infi} guarantees the difference between the strong chromatic index and its list version,
but it is not clear what is the exact value of the latter.
For the special case of the Petersen graph, we are able to prove the exact bound.
\begin{proof}[Proof of Theorem~\ref{thm:pet7}]
	We first prove that the list strong chromatic index of the Petersen graph $P$ is at least $7$.
	Consider the drawing of $P$ in Figure~\ref{fig:peter}.
	Let $L$ be the list assignment assigning the list $\set{1,2,4,5,7,8}$ to the outer cycle (the dashed edges),
	the list $\set{1,3,4,6,7,9}$ to the spokes (the dotted edges),
	and the list $\set{2,3,5,6,8,9}$ to the inner cycle (the solid edges).	
	\begin{figure}[htp!]
		$$
			\includegraphics{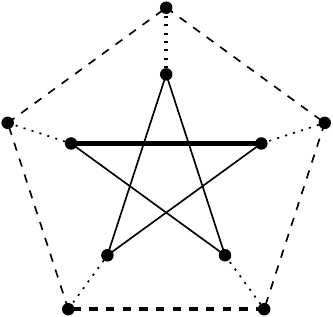}
		$$
		\caption{The Petersen graph $P$.}
		\label{fig:peter}
	\end{figure}
	Recall that every maximum induced matching in $P$ is of size $3$ and it contains precisely one edge of the outer cycle,
	one spoke, and one edge of the inner cycle (in Figure~\ref{fig:peter}, we depict one with bolder edges).
	Moreover, any pair of edges at distance $3$ belongs to exactly one maximum induced matching.
	%This means that at most three edges of $P$ can receive the same color.
	%Consequently, by the definition of $L$, any of the colors in the lists appears on at most two edges.
	Since there are five disjoint maximum induced matchings in $P$,
	one color can appear only on the edges of the same matching, 
	but on at most two of its edges.
	Hence, we need at least $5$ colors to color at most $10$ edges, and at least $5$ other colors to color 
	the remaining $5$ edges.
	However, we only have $9$ distinct colors in the union of lists of $L$, thus we cannot color the edges of $P$ from $L$.

	%Observe that the colors from $\set{1,2,3}$ can be used on at most five edges, belonging to two maximum induced matchings.
	%Similarly, the colors from $\set{4,5,6}$ can be used on at most five edges of another two maximum induced matchings.
	%The three edges of the remaining maximum induced matching, with no edge colored yet, can be colored with two colors, say $7$ and $8$.
	%However, the remaining color $9$ cannot be used to color the at least two non-colored edges of the first four colored maximum induced matchings,
	%since at least two edges are at distance at most $2$ (i.e., not belonging to a common induced matching).
	%Hence, $P$ cannot be strongly edge-colored from $L$, and so its list strong chromatic index is at least $7$.

	\medskip
	Now, we show that the list strong chromatic index of the Petersen graph $P$ is at most~$7$.
	Let $M_k$ denote the five disjoint maximum induced matchings in $P$ induced by the 
	edges $k_1,k_2,k_3$, for $k\in \{a, b, c, d, e\}$,
	with the labeling of the edges as shown in Figure~\ref{fig:peter1}.
	\begin{figure}[htp!]
		$$
			\includegraphics{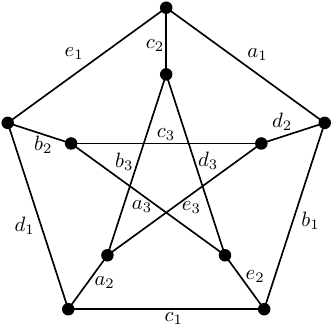}
		$$
		\caption{Five maximum induced matchings of the Petersen graph $P$.}
		\label{fig:peter1}
	\end{figure}

	In what follows, we will analyze the conflict graph polynomial $P_P$ of the Petersen graph.
	We first define an auxilliary polynomial (representing a conflict graph polynomial of two maximum induced matchings)
	\begin{align*}
		C(x_1,x_2,x_3,y_1,y_2,y_3) &= \prod_{i=1}^3 \prod_{j=1}^3 (x_i - y_j)\,.
	\end{align*}				
	Next, observe that only the edges of a particular matching can be colored by the same color, 
	and therefore each edge needs to receive a color distinct from colors of all other edges (from the other matchings).
	Hence, we have that
	\begin{align*}
		P_P(a_1, a_2, a_3, b_1, \dots, e_2, e_3) =& ~ C(a_1,a_2,a_3,b_1,b_2,b_3) \cdot C(a_1,a_2,a_3,c_1,c_2,c_3) \\
										    	  & \cdot C(a_1,a_2,a_3,d_1,d_2,d_3) \cdot C(a_1,a_2,a_3,e_1,e_2,e_3) \\
										    	  & \cdot C(b_1,b_2,b_3,c_1,c_2,c_3) \cdot C(b_1,b_2,b_3,d_1,d_2,d_3) \\
										    	  & \cdot C(b_1,b_2,b_3,e_1,e_2,e_3) \cdot C(c_1,c_2,c_3,d_1,d_2,d_3) \\
										    	  & \cdot C(c_1,c_2,c_3,e_1,e_2,e_3) \cdot C(d_1,d_2,d_3,e_1,e_2,e_3)\,.
	\end{align*}

	Now, we consider several cases regarding the possible colorings of the maximum induced matchings.
	Note that throughout the process of coloring the edges, as soon as some color is picked for an edge $e\in M_i$, this color is removed from the lists of the edges which are in conflict with $e$, 
	i.e., the edges of the maximum induced matchings different from $M_i$.

	\medskip
	\noindent {\bf Case~1.} \quad
	{\em Suppose that one maximum induced matching can be colored monochromatically.} 
		
		First, we color the edges of $M_a$, say by color $1$. 
		We distinguish three possible subcases regarding the coloring of the remaining maximum induced matchings.

		\medskip
	\noindent {\bf Case~1.1.} \quad
	{\em Suppose that one another maximum induced matching, say $M_b$, can be colored monochromatically (by color different from $1$).}

	Without loss of generality, we color $M_b$ by $2$.
	The remaining nine edges of $P$ have each at least $5$ colors available.
	Now, for the non-colored edges, we have the following conflict graph polynomial: 
	\begin{align*}
			P_{P^{-2}}(c_1, c_2, \dots, e_2, e_3) = & ~ C(c_1,c_2,c_3,d_1,d_2,d_3)  \cdot C(c_1,c_2,c_3,e_1,e_2,e_3)\\
														& \cdot C(d_1,d_2,d_3,e_1,e_2,e_3)\,.
		\end{align*}
	In $P_{P^{-2}}$, we have coefficient 
	$$
		\coef{P_{P^{-2}}}{c_1^4 c_2^3 c_3^3 d_1^3 d_2^3 d_3^3  e_1^3 e_2^3 e_3^2} = 94\,,
	$$ 
	which means, by Theorem~\ref{thm:null}, that it is possible to color the edges $c_1, c_2, c_3, d_1, d_2, d_3, e_1, e_2, e_3$ using the remaining colors of their lists.

	\medskip
	\noindent {\bf Case~1.2.} \quad
	{\em Suppose that no maximum induced matching except $M_a$ can be colored monochromatically, and at least one maximum induced matching, say 
	$M_b$, can be colored using exactly two colors.}		

Without loss of generality, we may assume that $M_b$ is colored by colors $2$ and $3$.
Since it is not possible to color any other maximum induced matching except $M_a$ with just one color (different from $1$), 
among the remaining edges, there is at least one such edge whose list does not contain both colors $2$ and $3$.
It follows that all edges have lists of size at least $4$, and at least one edge, say $c_1$, has list of size at least $5$.
For the non-colored edges, we again use the conflict graph polynomial as in Case~1, 
containing the coefficient 
$$
	\coef{P_{P^{-2}}}{c_1^4 c_2^3 c_3^3 d_1^3 d_2^3 d_3^3  e_1^3 e_2^3 e_3^2} = 94\,,
$$ 
which means that it is possible to color the non-colored edges using the colors from their lists.

	\medskip
	\noindent {\bf Case~1.3.} \quad
	{\em Suppose that no maximum induced matching except $M_a$ can be colored monochromatically, and no other maximum induced matching is colorable 
	by $2$ colors.}		

From now on we consider only the edges not included in $M_a$.
Each of the remaining maximum induced matchings consists of three edges with pairwise disjoint lists of colors of size at least $6$.
This means that the union of lists of any two edges of a maximum induced matching is of size at least $12$, and the union of lists of three edges of a 
maximum induced matching is of size at least $18$.
Consequently, for every set of at most $4$ edges it holds that the union of their color lists is of size at least $6$; 
for every set of $5$ to $8$ edges it holds that the union of their color lists is of size at least $12$; 
and for every set of $9$ to $12$ edges it holds that the union of their color lists is of size at least $18$.
Hence, we can apply Theorem \ref{thm:hall}, according to which it is possible to color all the remaining edges.

	\medskip
	\noindent {\bf Case~2.} \quad
	{\em Suppose that none of the maximum induced matchings can be colored by either $1$ or $2$ colors.}		

Thus, each of the maximum induced matchings consists of three edges with pairwise disjoint lists of colors of size $7$.
This means that the union of lists of any two edges of a maximum induced matching is of size $14$, 
and the union of lists of three edges of a maximum induced matching is of size $21$.
Therefore, for every set of at most five edges 
it holds that the union of their color lists is of size at least $7$; 
for every set of $6$ to $10$ edges it holds that the union of their color lists is of size at least $14$; 
and for every set of $11$ to $15$ edges it holds that the union of their color lists is of size at least $21$.
It follows that we can apply Theorem \ref{thm:hall}, and hence color all the edges by different colors.

\medskip
	\noindent {\bf Case~3.} \quad
	{\em Suppose that one maximum induced matching can be colored using $2$ colors (and none of them can be colored monochromatically).}	
	
	We color $M_a$, say by colors $1$ and $2$. 
	We consider two possible subcases regarding the coloring of the maximum induced matchings different from $M_a$.

	\medskip
	\noindent {\bf Case~3.1.} \quad
	{\em Suppose that none of the remaining maximum induced matchings can be colored using $2$ colors.}		

From now on, we only consider the edges not included in $M_a$.
Regarding any two (three) edges of any other maximum induced matching, we infer that the union of their color lists contains at least $10$ ($15$) colors.
%Consider one of the remaining maximum induced matchings, say $M_b$.
%Since no maximum induced matching is colorable by one color, and only $M_a$ is colorable by two colors, 
%the only colors which can occur multiple times in the lists of the edges of $M_b$ are colors $1$ and $2$.
%Without loss of generality, we may assume that color $1$ appears in the lists of the edges $b_1$ and $b_2$.
%It follows that color $2$ may appear either in the same two lists, or it appears in one common list, 
%say the one of the edge $b_1$, 
%and in the one of $b_3$.
%Hence, the reduced lists of the edges of the matching $M_b$ (after coloring $M_a$) 
%are either of size at least $5$, at least $5$, and at least $7$, 
%or of size at least $5$, at least $6$, and at least $6$.
%Clearly, the same argumentation applies to any other maximum induced matching.

Therefore, for every set of at most $4$ edges it holds that the union of their color lists is of size at least $5$; 
for every set of $5$ to $8$ edges it holds that the union of their color lists is of size at least $10$; 
and for every set of $9$ to $12$ edges it holds that the union of their color lists is of size at least $15$.
It follows that we can apply Theorem \ref{thm:hall}, and hence color all the remaining edges by different colors.

	\medskip
	\noindent {\bf Case~3.2.} \quad
	{\em Suppose that we can color at least one other maximum induced matching using $2$ colors.}

Without loss of generality, we may assume that $M_b$ is colored by colors $3$ and $4$.
Note that all these four colors can occur in the lists of other edges, but at most twice per maximum induced matching.
Regarding the setup of these colors, it follows that lists of edges of any particular maximum induced matching are of size: 
at least $3$, at least $3$, and at least $7$; 
or at least $3$, at least $4$, and at least $6$; 
or at least $3$, at least $5$, and at least $5$; 
or at least $4$, at least $4$, and at least $5$.
Note that these color lists may not be disjoint.

As above, for the non-colored edges we have the following conflict graph polynomial: 
\begin{align*}
		P_{P^{-2}}(c_1, c_2, \dots, e_2, e_3) = & ~ C(c_1,c_2,c_3,d_1,d_2,d_3)  \cdot C(c_1,c_2,c_3,e_1,e_2,e_3)\\
													& \cdot C(d_1,d_2,d_3,e_1,e_2,e_3)\,.
	\end{align*}
We consider the four cases regarding the sizes of the lists of the edges of maximum induced matchings as listed in the previous paragraph.

First, suppose there exists a maximum induced matching with lists of colors of the edges of sizes at least $3$, at least $3$, and at least $7$.
Then, in $P_{P^{-2}}$, we have the coefficients 
$$
	\coef{P_{P^{-2}}}{c_1^2 c_2^2 c_3^6 d_1^2 d_2^2 d_3^4  e_1^2 e_2^3 e_3^4} = -14\,,
$$
and
$$
	\coef{P_{P^{-2}}}{c_1^2 c_2^2 c_3^6 d_1^2 d_2^2 d_3^4  e_1^2 e_2^2 e_3^5} = -6\,,
$$ 
regarding the monomials fitting the possible sizes of the lists of remaining maximum induced matchings.
Therefore, by Theorem~\ref{thm:null}, it is possible to color the remaining edges.

If there is no maximum induced matching with the properties as in the previous case, 
then suppose that there is one with lists of colors of sizes at least $3$, at least $5$, and at least $5$.
Then, in $P_{P^{-2}}$, we have the coefficient 
$$
	\coef{P_{P^{-2}}}{c_1^2 c_2^4 c_3^4 d_1^2 d_2^3 d_3^4  e_1^2 e_2^2 e_3^4} = 60\,,
$$ 
which means that it is possible to color the remaining edges.

Now, we may assume that there exists no maximum induced matching which satisfies properties of previous cases.
Suppose that there is one maximum induced matching with lists of colors of sizes at least $3$, at least $4$, and at least $6$.
Then, in $P_{P^{-2}}$, there is the coefficient 
$$
	\coef{P_{P^{-2}}}{c_1^2 c_2^3 c_3^5 d_1^2 d_2^3 d_3^4  e_1^2 e_2^2 e_3^4} = 33\,,
$$  
which means that it is possible to color the remaining edges.

Lastly, suppose that all of the remaining maximum induced matchings have lists of colors of the edges of sizes at least $4$, at least $4$, and at least $5$.
Then, in $P_{P^{-2}}$, we have the coefficient 
$$
	\coef{P_{P^{-2}}}{c_1^3 c_2^3 c_3^4 d_1^2 d_2^3 d_3^4  e_1^1 e_2^3 e_3^4} = 36\,,
$$ 
which again means that it is possible to color the remaining edges.

Thus, the list strong chromatic index of $P$ is $7$.
\end{proof}

%%%%%%%%%%%%%%%%%%%%%%%%%%%%%%%%%%%%%%%%%%%%%%%%%%%%%%%%%%%%%%%%%%%%%%%%%%%%%
\section{Graphs $G$ with $\chin{G} < \chiln{G}$}
\label{sec:norlow}

In this section, we consider the results on list normal edge-coloring.
As already mentioned, lists of size at least $10$ are always enough 
to find a normal list edge-coloring of a cubic graph.
We do not know whether this bound is tight;
currently, there are only examples of graphs with list normal chromatic index equal to $9$.

\begin{theorem}
	\label{thm:nor9}	
	There is an infinite family of cubic graphs with list normal chromatic index at least $9$.
\end{theorem}

\begin{proof}
	In order to prove the theorem, we will show that if a cubic graph $G$ contains the configuration $H_I$ depicted in Figure~\ref{fig:9colors},
	then there is a list assignment for the edges of $H_I$, for which $G$ does not admit a list normal edge-coloring.
	We use the labeling of the vertices as given in the figure.
	\begin{figure}[htp!]
		$$
			\includegraphics{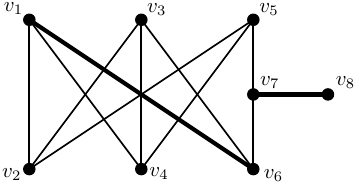}
		$$
		\caption{The configuration $H_I$ which is not list normal $8$-colorable if the two bold edges each receive a list disjoint 
			from all other lists in a given list assignment.}
		\label{fig:9colors}
	\end{figure}		

	Let $L$ be a list assignment for $G$ such that for every edge $e$ of $H_I$, except $v_1v_6$ and $v_7v_8$,
	we have $L(e) = [1,8]$. For the two special edges, we use $L(v_1v_6) = [9,16]$ and $L(v_7v_8) = [17,24]$.
	Without loss of generality, we may assume that $v_1v_6$ is colored by $9$, and $v_7v_8$ with $17$.
	There are nine remaining (thin) edges, which can altogether receive the eight distinct colors from $[1,8]$,
	so at least one pair must receive the same color. We will show that this is not possible.
	Note that the edges adjacent to the edges $v_1v_6$ and $v_7v_8$ must all be rich, 
	since they are adjacent to an edge with a unique color.

	First, without loss of generality, we color the edge $v_1v_2$ by $1$. 
	Since $v_1v_4$ must be rich, the edges $v_3v_4$ and $v_4v_5$ must be colored differently, with colors distinct from $1$,
	say with $2$ and $3$, respectively.

	Suppose now that the edge $v_3v_6$ also receives color 1. 
	Then, the edge $v_1v_6$ must be poor, meaning that the edges $v_1v_4$ and $v_6v_7$ receive the same color, say $4$. 			
	Similarly, the edge $v_2v_3$ must be poor, and so the edges $v_2v_5$ and $v_3v_4$ must be colored the same,
	which means that $v_2v_5$ also receives color $2$.			
	But now $v_4v_5$ must be poor, and thus $v_5v_7$ must be colored with $4$, which is not possible.

	Next, suppose that $v_3v_6$ is colored with $3$. 
	Then $v_3v_4$ must be poor and so the edges $v_1v_4$ and $v_2v_3$ must receive the same color,
	which is not possible, since $v_1v_2$ is rich.

	So, we may assume that $v_3v_6$ is colored with, say, $4$.
	Suppose first that $v_1v_4$ is colored with $4$. 
	Then, $v_3v_4$ is poor and $v_2v_3$ is colored with $3$.
	But now both $v_2v_5$ and $v_1v_6$ must be poor, 
	meaning that both $v_5v_7$ and $v_6v_7$ must be colored with $1$, a contradiction.	
	Therefore, we may assume that $v_1v_4$ is colored with a new color, say $5$,
	and consequently that $v_3v_4$ is rich, giving 
	that $v_2v_3$ receives a new color, say $6$.
	Since $v_1v_2$ and $v_2v_3$ are rich,
	a new color is given also to $v_2v_5$, say $7$.
	Finally, note that the edges adjacent to $v_5v_7$ are all rich and consequently it must be colored with $8$,
	which means that there is no available color for $v_6v_7$.
	Thus, for the given $L$, the graph $G$ does not admit a list normal edge-coloring, 
	and therefore $\chiln{G} > 8$.		
\end{proof}

%TODO: add that there is difference between neighbors? toto si nepamatam
In the above described family, every graph contains a bridge. 
As Conjecture~\ref{con:Nor} considers bridgeless cubic graphs only, it is natural to ask whether there are graphs
with the list normal chromatic index greater than $5$. 
The next example shows that even lists of size $7$ in the list assignment are sometimes not sufficient.
\begin{theorem}
	\label{thm:nor8}	
	There are bridgeless cubic graphs with list normal chromatic index at least $8$.
\end{theorem}

\begin{proof}
	As an example of a bridgeless cubic graph with the list normal chromatic index at least $8$, 
	we use the graph $G$ depicted in Figure~\ref{fig:8colors}.
	We will present a list assignment for the edges of $G$, for which $G$ does not admit a list normal $7$-edge-coloring.
	We use the labeling of the vertices as given in the figure.
	\begin{figure}[htp!]
		$$
			\includegraphics{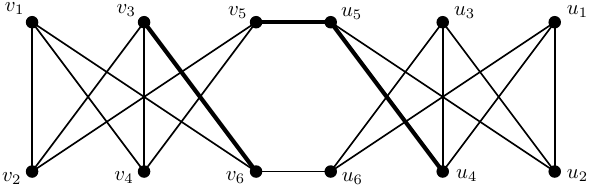}
		$$
		\caption{A bridgeless cubic graph which is not list normal $7$-colorable if the three bold edges each receive lists disjoint 
			from all other lists.}
		\label{fig:8colors}
	\end{figure}		
	
	Let $L$ be a list assignment for $G$ such that every edge $e$ of $G$, except $v_3v_6$, $v_5u_5$, and $u_4u_5$,
	has $L(e) = [1,7]$. 
	The three special edges have $L(v_3v_6) = [8,14]$, $L(v_5u_5) = [15,21]$, and $L(u_4u_5) = [22,28]$.
	Note that this setting implies that all the edges adjacent to these three edges must be rich.
	
	First, we color the three special edges; without loss of generality, 
	we color $v_3v_6$ with $8$, $v_5u_5$ with $15$, and $u_4u_5$ with $22$. 
	Next, we color the edges $v_1v_2$, $v_1v_4$, and $v_1v_6$ with, say, $1$, $2$, and $3$, respectively. 
	
	Now, the edge $v_4v_5$ cannot receive color $1$, since the edge $v_2v_5$ must be rich.
	If we color $v_4v_5$ with $3$, then $v_1v_4$ must be poor and $v_3v_4$ colored with $1$,
	which is not possible, since $v_2v_3$ must be rich.
	So, we assign to $v_4v_5$ color $4$.		
	Next, the edge $v_3v_4$ cannot receive any color from $\set{1, 2, 3, 4}$ (since $v_1v_4$ must be rich), 
	and thus we color it with $5$. 
	Since $v_3v_4$ must be rich, $v_2v_3$ cannot be colored with $2$,
	and since $v_4v_5$ must be rich, $v_2v_5$ cannot be colored with $2$.
	This means that $v_1v_2$ must be rich and consequently,
	$v_2v_3$ and $v_2v_5$ cannot receive any color from the set $\set{1, 2, 3, 4, 5}$.
	Therefore, we assign, say, color $6$ to $v_2v_3$ and color $7$ to $v_2v_5$. 
	
	At this point, the only possible colors the edge $v_6u_6$ can receive are $4$ and $7$. 	
	Note that these are exactly the colors already assigned to the edges adjacent to $v_5u_5$.
	
	\medskip
	Now consider the other non-colored edges of the graph. 
	Let $\set{a,b,c,d,e,f,g} = [1,7]$.	
	We can color the edges $u_1u_2$, $u_1u_4$, and $u_1u_6$ by three distinct colors, say $a$, $b$, and $c$, respectively. 
	Since $u_1u_4$ must be rich, $u_3u_4$ must be colored with a color distinct from the previous three, say with $d$. 	
	Since the edges $u_3u_4$ and $u_4u_5$ are rich, none of the edges $u_2u_3$ and $u_2u_5$ can be colored with $b$, 
	and therefore the edge $u_1u_2$ must be rich. 
	Moreover, $u_2u_5$ cannot receive color $d$, since $u_4u_5$ must be rich.
	Therefore, we color $u_2u_3$ by $e$ and $u_2u_5$ by $f$. 	
	The edge $u_3u_6$ cannot receive any color from the set $\{a, b, c, d, e, f\}$ 
	(since $u_2u_3$ and $u_3u_4$ must be rich), 
	so we must color it with the only remaining color $g$.
	Finally observe that the only possible color the edge $v_6u_6$ can receive is $f$,
	which is the same color as the color of $u_2u_5$. 
	This means that two of the edges adjacent to $v_5u_5$ (which must be rich) 
	must receive the same color, 
	this contradicts the fact that $v_5u_5$ must be rich.
\end{proof}

The graph in the proof of Theorem~\ref{thm:nor8} has a $2$-edge-cut.
So, again a question arises whether there are cubic graphs with high list normal chromatic index 
and high connectivity. 
We focused on cyclically $4$-edge-connected cubic graphs and surprisingly
there are such graphs with list normal chromatic index at least $7$.

\begin{theorem}
	\label{thm:nor7}	
	There is an infinite family of cyclically $4$-edge-connected cubic graphs with list normal chromatic index at least $7$.
\end{theorem}

\begin{proof}
	We will show that for a cubic graph $L_{2k}$ depicted in Figure~\ref{fig:9colors} and any $k \ge 5$,
	there is a list assignment for the edges of $L_{2k}$, for which $L_{2k}$ does not admit a list normal edge-coloring.
	Clearly, $L_{2k}$ is cyclically $4$-edge-connected.
	We use the labeling of the vertices as given in the figure.	
	\begin{figure}[htp!]		
		$$
			\includegraphics{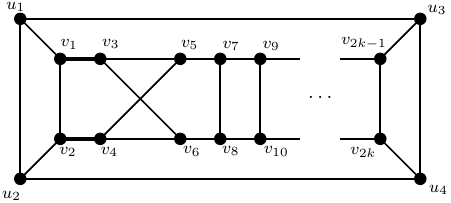}
		$$
		\caption{A cyclically $4$-edge-connected cubic graph $L_{2k}$ which is not list normal $6$-colorable 
			if the two bold edges receive lists disjoint from all other lists.}
		\label{fig:7colors-gadget}
	\end{figure}			

	Let $L$ be a list assignment for $L_{2k}$ such that for its every edge $e$, except $v_1v_3$ and $v_2v_4$,
	we have $L(e) = [1,6]$. For the two special edges, we use $L(v_1v_3) = [7,12]$ and $L(v_2v_4) = [13,18]$.
	Without loss of generality, we may assume that $v_1v_3$ is colored by $7$, and $v_2v_4$ with $13$.	
	
	Without loss of generality, we can assign colors $1$, $2$, and $3$ to the edges $v_3v_5$, $v_4v_5$, and $v_5v_7$, respectively. 
	Since the edges adjacent to the edges $v_1v_3$ and $v_2v_4$ must all be rich,	
	the edges $v_4v_6$ and $v_3v_6$ must obtain colors that were not used yet, say $4$ and $5$, respectively. 	
	Now, we consider two cases regarding the color of $v_6v_8$. 
	
	Suppose first that $v_6v_8$ is colored with $3$. 
	Then the edge $v_7v_8$ must be poor, 
	and consequently $v_7v_9$ and $v_8v_{10}$ must receive the same color.
	Now, following an analogous argument for coloring the remaining edges $v_iv_{i+1}$, $v_{i}v_{i+2}$, and $v_{i+1}v_{i+3}$, 
	for $i \in \set{9,11,\dots,2k-1}$, we infer that also $u_3u_4$ and $u_1u_2$ must be poor 
	and thus $u_1v_1$ and $u_2v_2$ must receive the same color.
	This is not possible, since $v_1v_2$ must be rich.
	
	So, $v_6v_8$ must be colored with color $6$ (the only color not used yet).
	Then, $v_7v_8$ must be rich and by the symmetry, we may assume that $v_7v_8$ is colored with $1$.
	It follows that $v_7v_9$ is colored with $2$,
	and $v_8v_{10}$ must be colored with $4$ or $5$.
	But this is not possible, since $v_6v_8$ must be rich.
	This completes the proof.
\end{proof}

\section{Conclusion}
\label{sec:conc}

One of the main results of this paper is the tight upper bound of $10$ colors for the strong chromatic index of subcubic graphs.
However, since this bound is only known to be attained by the Wagner graph and graph containing the $K_{3,3}$ with a subdivided edge as a subgraph,
one may ask whether there are other examples of such graphs, perhaps with smaller chromatic index.
\begin{question}
	Is it true that for any subcubic graph $G$ with $\chils{G} = 10$, 
	it holds that $\chis{G} = 10$?
\end{question}

As the second main result, we proved that the strong chromatic index and the list strong chromatic index differ for some graphs;
we provided an infinite family of such graphs, 
but the family only contains graphs with the minimum possible value of the strong chromatic index,
and it does not seem likely that for graphs with strong chromatic index closer to the general upper bound of $10$ colors,
their list strong chromatic index will be different.
Therefore, we propose a rather bold statement, which is in line with Conjecture~\ref{con:str8}.
\begin{conjecture}
	\label{con:liststr8}
	For any connected bridgeless subcubic graph $G$ on at least $13$ vertices, it holds that
	$$
		\chils{G} \le 8\,.
	$$	
\end{conjecture}
The first step towards proving this conjecture would be proving that the list strong chromatic index
of any connected bridgeless subcubic graph, not isomorphic to the Wagner graph, is at most $9$.
Or even more specifically, finding the exact upper bounds for the list strong chromatic indices of special graph families 
such as planar graphs and bipartite graphs would also give a relevant insight into the topic.

On the other hand, we do believe that there are graphs with strong chromatic index $6$ 
and greater list strong chromatic index.
\begin{problem}
	Find an infinite family of graphs $G$ with $\chis{G} = 6$ and 
	$$
		\chils{G} > \chis{G}\,.
	$$
\end{problem}

Also, we are confident that Theorem~\ref{thm:infi} can be extended to all cubic graphs
with strong chromatic index equal to $5$.
\begin{conjecture}
	For every cubic graph $G$ with $\chis{G} = 5$ we have that $\chils{G} > 5$.
\end{conjecture}

The strong edge-coloring is an important concept; the study of (sub)cubic graphs is popular as these graphs
are somewhat easier to handle than general graphs.
The properties of list strong edge-coloring for general graphs are thus also of interest.
In~\cite{LuzMacSkoSot22}, it was shown that $k$-regular graphs attaining the lowest possible value $2k-1$ of the strong chromatic index
are precisely the covers of the Kneser graphs $K(2k-1,k)$.
It seems that Theorem~\ref{thm:infi}, with some additional effort, could be extended to regular graphs of greater degree.
Along these lines, we suggest the following question.
\begin{question}
	Is it true that for a given integer $k \ge 4$, there is an infinite family of graphs $G$ of maximum degree $k$
	such that 
	$$
		\chils{G} > \chis{G}\,?
	$$
\end{question}

\bigskip
In the case of list normal edge-coloring, we have the upper bound given by Theorem~\ref{thm:lst10},
but we do not have an example of a graph attaining the bound;
in fact, we do not believe one exists.
\begin{conjecture}
	For any cubic graph $G$, it holds that 
	$$
		\chiln{G} \le 9\,.
	$$
\end{conjecture}

Conjecture~\ref{con:Nor} assumes only bridgeless cubic graphs. 
We showed in Theorem~\ref{thm:nor8} that in the list version, 
there are bridgeless cubic graphs with list normal chromatic index at least $8$.
In the proof of the theorem, we only provided one graph of order $12$. 
However, we are only aware of two other graphs with list normal chromatic index at least~$8$; namely the Wagner graph 
and the graph obtained from $K_{3,3}$ in which one vertex is truncated (see Figure~\ref{fig:wagner}).
We also remark here without a proof that with some additional effort, 
one can show that the list normal chromatic index of the Wagner graph is equal to $8$.
\begin{figure}[htp!]		
	$$
		\includegraphics{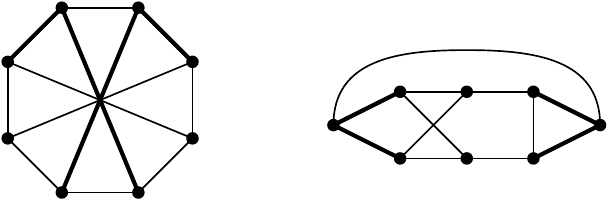}
	$$
	\caption{Two bridgeless cubic graphs with list strong chromatic index at least $8$.
		If the bold edges receive unique lists of seven colors and all the other edges the same lists of seven colors,
		then one can not realize a list normal edge-coloring.}
	\label{fig:wagner}
\end{figure}

Based on our results and additional computer tests on small graphs, 
we confidently propose also the following.
\begin{conjecture}
	\label{con:listnor7}
	For any connected bridgeless cubic graph $G$ on at least $14$ vertices, it holds that
	$$
		\chiln{G} \le 7\,.
	$$	
\end{conjecture}

As opposed to the normal edge-coloring, in its list version, 
the property of being a class I graph does not resolve the problem trivially. 
In fact, it seems that the following are highly non-trivial questions.
\begin{question}
	What is the tight upper bound for the list normal chromatic index of a cubic graph, which is:
	\begin{enumerate}[(a)]
		\item{} (bridgeless) planar;
		\item{} class I;
		\item{} bipartite;
		\item{} with girth at least $C$, for some large enough constant $C$;
		\item{} cyclically $k$-edge-connected, for some integer $k$?
	\end{enumerate}
\end{question}

%%%%%%%%%%%%%%%%%%%%%%%%%%%%%%%%%%%%%%%%%%%%%%%%%%%%%%%%%%%%%%%%%%%%%%%%%%%%%
\paragraph{Acknowledgement.} 
B.~Lu\v{z}ar was supported by the Slovenian Research and Innovation Agency Program P1--0383 and the projects J1--3002 and J1--4008.
E. M\'{a}\v{c}ajov\'{a} was partially supported by the Slovak Research and Development Agency under the contract APVV--23--0076.
R. Sot\'{a}k and D. \v{S}vecov\'{a} were supported by the Slovak Research and Development Agency under the contracts No. APVV--19--0153 and APVV--23--0191, 
and the VEGA Research Grant 1/0574/21.
	
%%%%%%%%%%%%%%%%%%%%%%%%%%%%%%%%%%%%%%%%%%%%%%%%%%%%%%%%%%%%%%%%%%%%%%%%%%%%%
\bibliographystyle{plain}
{
	\bibliography{References}

\begin{thebibliography}{10}

\bibitem{Alo99}
N.~Alon.
\newblock {Combinatorial Nullstellensatz}.
\newblock {\em Combin. Probab. Comput.}, 8(1--2):7--29, 1999.

\bibitem{And92}
L.~D. Andersen.
\newblock The strong chromatic index of a cubic graph is at most 10.
\newblock {\em Discrete Math.}, 108:231--252, 1992.

\bibitem{BonDelLanPos24}
M.~Bonamy, M.~Delcourt, R.~Lang, and L.~Postle.
\newblock {Edge-colouring graphs with local list sizes}.
\newblock {\em J. Combin. Theory Ser. B}, 165:68--96, 2024.

\bibitem{CheHuYuZho19}
M.~Chen, J.~Hu, X.~Yu, and S.~Zhou.
\newblock {List strong edge coloring of planar graphs with maximum degree 4}.
\newblock {\em Discrete Math.}, 342:1471--1480, 2019.

\bibitem{DaiWanYanYu18}
T.~Dai, G.~Wang, D.~Yang, and G.~Yu.
\newblock {Strong list-chromatic index of subcubic graphs}.
\newblock {\em Discrete Math.}, 341:3434--3440, 2018.

\bibitem{Erd88}
P.~Erd{\H{o}}s.
\newblock {Problems and Results in Combinatorial Analysis and Graph Theory}.
\newblock In J.~Akiyama, Y.~Egawa, and H.~Enomoto, editors, {\em Graph Theory
  and Applications}, volume~38 of {\em Annals of Discrete Mathematics}, pages
  81--92. Elsevier, 1988.

\bibitem{FabLuzSotSve24}
I.~Fabrici, B.~Lu\v{z}ar, R.~Sot\'{a}k, and D.~\v{S}vecov\'{a}.
\newblock {Normal $6$-edge-colorings of cubic graphs with oddness $2$}, 2024.
\newblock Manuscript.

\bibitem{FauGyaSchTuz90}
R.~J. Faudree, A.~Gy\'{a}rf\'{a}s, R.~H. Schelp, and Z.~Tuza.
\newblock The strong chromatic index of graphs.
\newblock {\em Ars Combin.}, 29B:205--211, 1990.

\bibitem{FerMazMkr20}
L.~Ferrarini, G.~Mazzuoccolo, and V.~Mkrtchyan.
\newblock {Normal 5-edge-colorings of a family of Loupekhine snarks}.
\newblock {\em AKCE Int. J. Graphs Comb.}, 17:720--724, 2020.

\bibitem{Ful71}
D.~R. Fulkerson.
\newblock {Blocking and anti-blocking pairs of polyhedra}.
\newblock {\em Math. Program.}, 1(1):168--194, 1971.

\bibitem{HagSte13}
J.~H\"{a}gglund and E.~Steffen.
\newblock {Petersen-colorings and some families of snarks}.
\newblock {\em Ars Math. Contemp.}, 7:161--173, 2013.

\bibitem{Hal35}
P.~Hall.
\newblock {On Representatives of Subsets}.
\newblock {\em J. Lond. Math. Soc.}, s1--10(1):26--30, 1935.

\bibitem{Has22}
M.~Hasanvand.
\newblock {The List Square Coloring Conjecture fails for bipartite planar
  graphs and their line graphs}.
\newblock 2025.
\newblock ArXiv Preprint no. 2211.00622.

\bibitem{HorQinTro93}
P.~Hor{\'a}k, H.~Qing, and W.~T. Trotter.
\newblock Induced matchings in cubic graphs.
\newblock {\em J.~Graph Theory}, 17:151--160, 1993.

\bibitem{HorWoz12}
M.~Hor{\v{n}}{\'a}k and M.~Wo{\'z}niak.
\newblock On neighbour-distinguishing colourings from lists.
\newblock {\em Discrete Math. Theor. Comput. Sci.}, 14:21--28, 2012.

\bibitem{HurJoaKan21}
E.~Hurley, R.~de~Joannis~de Verclos, and R.~J. Kang.
\newblock {An Improved Procedure for Colouring Graphs of Bounded Local
  Density}.
\newblock In {\em Proceedings of the Thirty-Second Annual ACM-SIAM Symposium on
  Discrete Algorithms}, SODA '21, pages 135--148, USA, 2021. Society for
  Industrial and Applied Mathematics.

\bibitem{Mat24}
Wolfram~Research{,} Inc.
\newblock {Mathematica, {V}ersion 14.0}.
\newblock Champaign, IL, 2024.

\bibitem{Jae85b}
F.~Jaeger.
\newblock {A Survey of the Cycle Double Cover Conjecture}.
\newblock In B.R. Alspach and C.D. Godsil, editors, {\em Annals of Discrete
  Mathematics (27): Cycles in Graphs}, volume 115 of {\em North-Holland
  Mathematics Studies}, pages 1--12. North-Holland, 1985.

\bibitem{Jae85}
F.~Jaeger.
\newblock On five-edge-colorings of cubic graphs and nowhere-zero flow
  problems.
\newblock {\em Ars Combin.}, 20B:229--244, 1985.

\bibitem{Jae88}
F.~Jaeger.
\newblock Nowhere-zero flow problems.
\newblock In {\em Selected topics in graph theory}, pages 71--95. Academic
  Press, 1988.

\bibitem{JenTof95}
T.R. Jensen and B.~Toft.
\newblock {\em {Graph Coloring Problems}}.
\newblock Wiley Series in Discrete Mathematics and Optimization. Wiley, New
  York, 1995.

\bibitem{KarMacZer23}
F.~Kardo\v{s}, E.~M\'{a}\v{c}ajov\'{a}, and J.~P. Zerafa.
\newblock {Disjoint odd circuits in a bridgeless cubic graph can be quelled by
  a single perfect matching}.
\newblock {\em J. Combin. Theory Ser. B}, 160:1--14, 2023.

\bibitem{Kie00}
H.A. Kierstead.
\newblock On the choosability of complete multipartite graphs with part size
  three.
\newblock {\em Discrete Math.}, 211:255--259, 1 2000.

\bibitem{KimPar15}
S.~J. Kim and B.~Park.
\newblock {Counterexamples to the list square coloring conjecture}.
\newblock {\em J. Graph Theory}, 78:239--247, 2015.

\bibitem{KosWoo01}
A.~V. Kostochka and D.~R. Woodall.
\newblock {Choosability conjectures and multicircuits}.
\newblock {\em Discrete Math.}, 240(1):123--143, 2001.

\bibitem{LuzMacSkoSot22}
B.~Lu\v{z}ar, E.~M\'{a}\v{c}ajov\'{a}, M.~\v{S}koviera, and R.~Sot\'{a}k.
\newblock {Strong edge colorings of graphs and the covers of Kneser graphs}.
\newblock {\em J. Graph Theory}, 100:686--697, 2022.

\bibitem{MazMkr20b}
G.~Mazzuoccolo and V.~Mkrtchyan.
\newblock {Normal 6-edge-colorings of some bridgeless cubic graphs}.
\newblock {\em Discrete Appl. Math.}, 277:252--262, 2020.

\bibitem{MazMkr20}
G.~Mazzuoccolo and V.~Mkrtchyan.
\newblock {Normal edge-colorings of cubic graphs}.
\newblock {\em J. Graph Theory}, 94:75--91, 5 2020.

\bibitem{PanHuoCheWan21}
X.~Pang, J.~Huo, M.~Chen, and W.~Wang.
\newblock {A note on strong edge choosability of toroidal subcubic graphs}.
\newblock {\em Australas. J. Combin.}, 79:284--294, 2021.

\bibitem{SedSkr24b}
J.~Sedlar and R.~\v{S}krekovski.
\newblock {Normal 5-edge-coloring of some snarks superpositioned by Flower
  snarks}.
\newblock {\em European J. Combin.}, 122:104038, 2024.

\bibitem{SedSkr24}
J.~Sedlar and R.~\v{S}krekovski.
\newblock {Normal 5-edge-coloring of some snarks superpositioned by the
  Petersen graph}.
\newblock {\em Appl. Math. Comput.}, 467:128493, 2024.

\bibitem{ZhaChaHuMaYan20}
B.~Zhang, Y.~Chang, J.~Hu, M.~Ma, and D.~Yang.
\newblock {List strong edge-coloring of graphs with maximum degree 4}.
\newblock {\em Discrete Math.}, 343:111854, 2020.

\end{thebibliography}
}

\end{document}